\documentclass[12pt]{amsart}
\usepackage{amsfonts}
\usepackage{amsthm}
\usepackage{amsmath}
\usepackage{amssymb}
\usepackage{mathtools}
\usepackage{amscd}
\usepackage[latin2]{inputenc}
\usepackage{t1enc}
\usepackage[mathscr]{eucal}
\usepackage{indentfirst}
\usepackage{graphicx}
\usepackage{graphics}
\usepackage{pict2e}
\usepackage{epic}
\numberwithin{equation}{section}
\usepackage[margin=3.5cm]{geometry}
\usepackage{epstopdf} 
\usepackage{commath}

\theoremstyle{plain}
\newtheorem{Th}{Theorem}[section]
\newtheorem{Lem}[Th]{Lemma}

\theoremstyle{definition}

\newtheorem{Que}{Question}

\newtheorem{?}[Th]{Problem}

\begin{document}

\title{Homology handles with trivial Alexander polynomial}

\author{Dongsoo Lee}

\address{Department of Mathematics, Michigan State University, East Lansing, MI, 48824} 

\email{leedon56@msu.edu}

 \subjclass[2010]{57N70, 57M10.}

 \keywords{Homology handles, $\widetilde{H}$-cobordism} 

\begin{abstract}
Using Freedman and Quinn's result for $\mathbb{Z}$-homology 3-spheres, we show that a 3-dimensional homology handle with trivial Alexander polynomial bounds a homology $S^1\times D^3$. As a consequence, a distinguished homology handle with trivial Alexander polynomial is topologically null $\widetilde{H}$-cobordant. 

\end{abstract}

\maketitle

\section{Introduction}
In 1976, Kawauchi introduced the smooth $\widetilde{H}$-cobordism group $\Omega(S^1\times S^2)$, whose elements are equivalence classes of distinguished homology handles.  A {\em distinguished homology handle} is a pair $(M, \alpha)$ of a compact, oriented 3-manifold $M$ having the homology of $S^1\times S^2$, and a specified generator $\alpha$ of $H_1(M;\mathbb{Z})\cong \mathbb{Z}$. The equivalence relation is $\widetilde{H}$-cobordism, which means that two distinguished homology handles $(M_0, \alpha_0)$ and $(M_1, \alpha_1)$ are {\em $\widetilde{H}$-cobordant} if there is a pair $(W, \varphi)$ of a smooth connected 4-dimensional cobordism $W$ between $M_0$ and $M_1$, and a first cohomology class $\varphi \in H^1(W;\mathbb{Z})$ such that 
\begin{enumerate}
\item $\varphi_{|_{M_i}}$ are dual to $\alpha_i$ for $i=0, 1$,
\item $H_*(\widetilde{W}_\varphi;\mathbb{Q})$ is finitely generated over $\mathbb{Q}$ for each $*$, where $\widetilde{W}_\varphi$ is the infinite cyclic covering of $W$ associated with $\varphi$.
\end{enumerate}
In this case, $(W, \varphi)$ (or simply $W$) is called a {\em smooth $\widetilde{H}$-cobordism} between $(M_0, \alpha_0)$ and $(M_1, \alpha_1)$ (or between $M_0$ and $M_1$).
If $(M, \alpha)$ is $\widetilde{H}$-cobordant to $(S^1\times S^2, \alpha_{\scriptscriptstyle S^1\times S^2})$, where $\alpha_{\scriptscriptstyle S^1\times S^2}$ is the homology class of $S^1\times*$ with a fixed orientation, then we say that $(M, \alpha)$ is {\em null $\widetilde{H}$-cobordant}, and $(W^+, \varphi)$ (or $W^+$) is a {\em null $\widetilde{H}$-cobordism} of $(M, \alpha)$ (or of $M$). Equivalently, there is a smooth $\widetilde{H}$-cobordism $(W^+, \varphi)$ with $\partial W^+ = M$. Under a sum operation $\bigcirc$ called the {\em circle union}, $\Omega(S^1\times S^2)$ is an abelian group, and $[(S^1\times S^2, \alpha_{\scriptscriptstyle S^1\times S^2})]$ plays the role of the identity. Furthermore, the inverse $-[(M, \alpha)]$ of $[(M, \alpha)]$ is $[(-M, \alpha)]$, where $-M$ is $M$ with a reversed orientation. For details, see \cite{K1} and \cite{Lee}.\

\smallskip
Likewise, we can define the topological $\widetilde{H}$-cobordism group $\Omega^{top}(S^1\times S^2)$ in the topological category by using topological 4-manifolds in the definition of $\widetilde{H}$-cobordism. There is a natural surjective homomorphism $\psi:\Omega(S^1\times S^2) \to \Omega^{top}(S^1\times S^2)$ by forgetting smooth structures.\

\smallskip

Results in knot concordance motivate a number of questions on the $\widetilde{H}$-cobordism groups $\Omega(S^1\times S^2)$ and $\Omega^{top}(S^1\times S^2)$.\

In the knot concordance group $\mathcal{C}$, let $\mathcal{C}_\Delta$ be the subgroup generated by knots with trivial Alexander polynomial, and $\mathcal{C}_T$ the subgroup generated by topologically slice knots. Using Donaldson's diagonalization theorem \cite{D}, Casson observed that there are knots with trivial Alexander polynomial but which are not smoothly slice (appearing in \cite{CG}). After Donaldson's result, Freedman proved that a knot with trivial Alexander polynomial is topologically slice \cite{F}, \cite{FQ}. Thus $\mathcal{C}_{\Delta}\subset\mathcal{C}_{T}$ and $\mathcal{C}_{T}$ is non trivial, i.e., the map $\mathcal{C}\to \mathcal{C}^{top}$ is not injective, where $\mathcal{C}^{top}$ is the topologically flat knot concordance group.\
\smallskip

One can expect similar results in the $\widetilde{H}$-cobordism groups. Let $\Omega_{\Delta}$ be the subgroup generated by distinguished homology handles with trivial Alexander polynomial, and $\Omega_T$ the kernel of the map $\psi: \Omega(S^1\times S^2) \to \Omega^{top}(S^1\times S^2)$.

\begin{Que}\label{Q1}
Is $\psi:\Omega(S^1\times S^2)\to\Omega^{top}(S^1\times S^2)$ injective?
\end{Que}

\begin{Que}\label{Q2}
Is $\Omega_{\Delta} \subset \Omega_{T}$?
\end{Que}

In this paper, using the work of Freedman and Quinn, we prove the following theorem, which is the positive answer to Question \ref{Q2}.

\newtheorem*{MT}{Theorem 1}\label{main}
\begin{MT}
A distinguished homology handle with trivial Alexander polynomial is topologically null $\widetilde{H}$-cobordant.
\end{MT}

We expect a negative answer to Question \ref{Q1}. Then one can also ask about $\Omega_{T}/\Omega_{\Delta}$.
\begin{Que}
How big is the gap between two groups if $\Omega_{T}/\Omega_{\Delta}$ is non-trivial?
\end{Que}

In the knot concordance group $\mathcal{C}$, Hedden, Livingston, and Ruberman showed that $\mathcal{C}_T/\mathcal{C}_{\Delta}$ contains a $\mathbb{Z}^{\infty}$-subgroup \cite{HLR}, and Hedden, Kim and Livingston showed that it also has a $\mathbb{Z}_2^{\infty}$-subgroup \cite{HKL}.

\section*{Acknowledgement}
I would like to thank my advisor Matthew Hedden for an enlightening email discussion.

\section{Alexander polynomial of homology handles}\label{s2}
In this section, we review Alexander polynomial of homology handles. We refer the reader to \cite{K1}, \cite{K2}, and \cite{M1} for more details.\

\smallskip

Let $M$ be an oriented homology handle. Then we have the infinite cyclic covering $\widetilde{M}$ of $M$ associated with the abelianization map $\pi_1(M) \twoheadrightarrow H_1(M)\cong \mathbb{Z}$. Let $t$ be a generator of the deck transformation group $\mathbb{Z}$ of the covering space. Since $M$ is compact and triangulable, it admits a finite CW-complex, and thus the chain complex $C_i(\widetilde{M};\mathbb{Z})$ can be considered as a free and finitely generated module over the group ring $\Lambda=\mathbb{Z}[\mathbb{Z}]=\mathbb{Z}[t,t^{-1}]$, with one generator for each $i$-cell of $M$. Since the group ring $\Lambda$ is Noetherian, one can see that the homology $H_i(\widetilde{M};\mathbb{Z})$ is a finitely generated module over $\Lambda$. For an exact sequence $E\to F\to H_1(\widetilde{M};\mathbb{Z})\to 0$ of $\Lambda$-modules with $E$ and $F$ free modules of finite rank, a {\em presentation matrix} $P$ is a matrix representing the homomorphism $E\to F$. If the rank of $F$ is $r\ge1$, then the {\em first elementary ideal} $\mathcal{E}$ of $P$ is the ideal over $\Lambda$ generated by all the $r\times r$ minors of $P$. If there are no $r\times r$ minors, then we have $\mathcal{E}=0$, and if $r=0$, then we set $\mathcal{E}=\Lambda$. The {\em Alexander polynomial} of $M$ is defined to be any generator $\bigtriangleup_{M}(t)$ of the smallest principal ideal over $\Lambda$ containing $\mathcal{E}$.\

\smallskip

{\em Another description}. Let $\mu$ be a smoothly embedded simple closed oriented curve in $M$ representing a generator of $H_1(M;\mathbb{Z})$. Let $T(\mu)$ be a tubular neighborhood of $\mu$. We choose simple closed oriented smooth curves $K$ and $l$ in $\partial T(\mu)$ intersecting in a single point so that $l$ is homologous to $\mu$ in $T(\mu)$, and $K$ bounds a disk in $T(\mu)$ with ${\rm lk}(\mu,K)=+1$. Note that the choice of a curve $l$ is not unique. Choose a diffeomorphism $h:S^1\times S^1\to \partial T(\mu)$ such that $h(S^1\times0)=l$ and $h(0\times S^1)=K$. Let $Y=(M\setminus {\rm Int}T(\mu))\cup_{h}(D^2\times S^1)$. Then $Y$ is a  $\mathbb{Z}$-homology 3-sphere, and $K$ is a knot in $Y$. The {\em Alexander polynomial} $\bigtriangleup_{M}(t)$ of $M$ is defined to be the Alexander polynomial $\bigtriangleup_{K}(t)$ of $K$ in $Y$.\

\smallskip

Both definitions agree with the following: Let $A$ be a Seifert matrix for a knot $K$ in $Y$. We know that $tA-A^{T}$ is a presentation matrix for the $\Lambda$-module $H_1(\widetilde{X(K)},\mathbb{Z})$, where $X(K)$ is a knot exterior of $K$ in $Y$, and $\widetilde{X(K)}$ is the infinite cyclic coverings of $X(K)$. Let $Y_0(K)$ be the 3-manifold obtained from $Y$ by 0-surgery along $K$ in $Y$. Then we have a canonical isomorphism $H_1(\widetilde{X(K)};\mathbb{Z})\cong H_1(\widetilde{Y_0(K)};\mathbb{Z})$. In fact, $Y_0(K)\cong M$ as the two surgeries along $\mu$ and $K$ are dual to each other. So, $tA-A^{T}$ is also a presentation matrix for the $\Lambda$-module $H_1(\widetilde{M};\mathbb{Z})$. The matrix $tA-A^{T}$ is a square matrix, so by definition $\bigtriangleup_{M}(t)={\rm det}(tA-A^{T})=\bigtriangleup_{K}(t)$.

\section{proof of Theorem \ref{main}}
{\em Throughout this section, homologies are all over $\mathbb{Z}$.}\

\smallskip
Let $M$ be an oriented homology handle, so that its homology groups are isomorphic to those of $S^1 \times S^2$, and suppose that $\bigtriangleup_{M}(t)=1$. We will use the same notation as Section \ref{s2}. By attaching a 2-handle $D^2\times D^2$ to the boundary $M\times 0$ of $M\times [0,1]$ along $\mu$ with a framing determined by the curve $l$, we obtain a cobordism $X=(M\times [0,1])\cup_{l-{\rm framing}}(D^2\times D^2)$ between $Y$ and $M$. Then $K$ is a knot in $Y$ with Alexander polynomial $\bigtriangleup_{K}(t)=\bigtriangleup_{M}(t)=1$. By \cite[11.7B Theorem]{FQ}, there is a pair $(W^{'}, D)$, where $W^{'}$ is a contractible topological 4-manifold, and $D$ is a locally flat 2-disk properly embedded in $W^{'}$ such that $\partial(W^{'}, D)=(Y, K)$. By stacking $X$ to $W^{'}$ along $Y$, we obtain a topological 4-manifold $W^{''}=X\cup_{Y}W^{'}$, which $M$ bounds. Furthermore, we obtain a locally flat 2-sphere $S$ in $W^{''}$ from the union of the cocore of the 2-handle and the locally flat 2-disk $D$.

\begin{Lem}\label{Lem1}
The 4-manifold $W^{''}$ has the homology of $D^2\times S^2$.
\end{Lem}
\begin{proof}
First, we compute the homology of $X$, which is obtained from $M\times [0,1]$ by attaching a 2-handle $D^2\times D^2$. The attaching region is a tubular neighborhood of $\mu$, and is homeomorphic to $S^1\times D^2$. From the Mayer-Vietoris sequence, we have the following: 
\begin{alignat*}{2} \tag{1}
\cdots & \to H_{i}(S^1\times D^2)\to H_{i}(M\times [0,1]) & {}\oplus H_{i}(D^2\times D^2) & \to H_{i}(X) \\
       & \to H_{i-1}(S^1\times D^2)\to\cdots  & \cdots          & \to H_{0}(X)\to 0.
\end{alignat*}
Note that $H_1(S^1\times D^2)\to H_1(M\times[0,1])$ is an isomorphism and $H_0(S^1\times D^2)\to H_0(M\times[0,1])\oplus H_0(D^2\times D^2)$ is injective. Then it is easy to find that $H_i(X) \cong \mathbb{Z}$ if $i=0, 2, 3$, and trivial otherwise.\

Next, we compute the homology of $W^{''}$ using the Mayer-Vietoris sequence as follows. Since the intersection between $X$ and $W^{'}$ is $Y$, we have the following: 
\begin{alignat*}{2}
\cdots & \to H_{i}(Y)\to H_{i}(X) & {}\oplus H_{i}(W^{'}) & \to H_{i}(W^{''}) \\
       & \to H_{i-1}(Y)\to\cdots  & \cdots          & \to H_{0}(W^{''})\to 0.
\end{alignat*}
Note that $H_i(Y)\cong H_i(S^3)$ and $H_i(W^{'})\cong H_i(B^4)$. Since the map $H_0(Y) \to H_0(X)\oplus H_0(W^{'})$ is injective, we have $H_0(W^{''})\cong\mathbb{Z}$, $H_1(W^{''})\cong 0$, and $H_2(W^{''})\cong \mathbb{Z}$. Considering the maps $H_3(Y)\to H_3(X) \gets H_3(M)$ induced by inclusions, the right map is an isomorphism from the long exact sequence (1), and the images of two maps are homologous in $H_3(X)$. Then the left map is also an isomorphism, and thus $H_3(W^{''})\cong H_4(W^{''})\cong 0$.\

\end{proof}

\begin{Lem}\label{Lem2}
The locally flat 2-sphere $S$ represents a generator of $H_2(W^{''})$, and its self-intersection $S\cdot S$ is $0$.
\end{Lem}
\begin{proof}
Let $\alpha$ be the 2-disk obtained from the union of $\mu\times[0,1]$ and the core of the 2-handle. Then its boundary is $\mu$ in $\partial W^{''}=M$, and it intersects $S$ in a single point. Thus, to show that $S$ represents a generator of $H_2(W^{''})$, it suffices that $\alpha$ represents a generator of a $\mathbb{Z}$-summand of $H_2(W^{''},\partial W^{''})$. Note that $H_2(M)\cong H_2(M\times [0,1])\cong H_2(X)\cong H_2(W^{''})$ from long exact sequences in the proof of Lemma \ref{Lem1}. We consider the long exact sequence of the pair $(W^{''}, \partial W^{''})$:\ 
$$H_2(M)\to H_2(W^{''})\xrightarrow{[A]} H_2(W^{''}, M) \xrightarrow{\partial} H_1(M)\to 0.$$
It is well-known that the map $[A]$ is represented by an intersection form $A$ of $H_2(W^{''})$ with respect to some basis since $\partial W^{''}\neq \emptyset$ and $H_1(W^{''})$ is trivial, see \cite[\S3]{GL}.
Because the first map is an isomorphism, the intersection form $A$ is trivial and $H_2(W^{''}, M)\cong H_1(M)\cong \mathbb{Z}$. Since $\partial[\alpha]=[\mu]$ and $[\mu]$ is a generator of $H_1(M)$, $\alpha$ represents a generator of $H_2(W^{''}, M)$. \
\end{proof}

Since $S$ is locally flat, it has a normal bundle by \cite[\S9.3]{FQ}, and hence it has a tubular neighborhood $T(S)$ in $W^{''}$. The normal bundle over $S$ is determined by its Euler number, which equals the algebraic intersection number between the 0-section and any other section transverse to it. By Lemma \ref{Lem2}, $S\cdot S=0$. Thus, the normal bundle is trivial, and $T(S)$ is homeomorphic to $S^2\times D^2$. Let $\psi:S^2\times D^2 \to W^{''}$ be an embedding of the tubular neighborhood of $S$. Let $X(S)$ be the exterior of the sphere $S$, i.e., $X(S)=W^{''}\backslash{\rm Int}T(S)$. Let $W$ be the 4-manifold obtained from $X(S)$ by gluing in $D^3\times S^1$ back along the boundary $S^2\times S^1$ of $X(S)$. That is, $W=X(S)\cup_{\psi|_{(S^2\times S^1)}}(D^3\times S^1)$.

\begin{Lem}
The 4-manifold $W$ has the homology of $D^3\times S^1$.
\end{Lem}
\begin{proof}
The homology exact sequence of the pair $(W^{''}, X(S))$ yields:
\begin{alignat*}{10}
\cdots &\to H_{i}(X(S))\to H_{i}(W^{''})&\to H_{i}(W^{''}, X(S))&\\
            &\to H_{i-1}(X(S))\to~~\cdots &\to H_{0}(W^{''}, X(S))&\to 0.
\end{alignat*}
Via excision, $H_i(W^{''}, X(S))\cong H_i(T(S), \partial T(S))\cong \mathbb{Z}$ for $i=2, 4$, and trivial otherwise. Then we can easily obtain $H_0(X(S))\cong H_3(X(S))\cong \mathbb{Z}$ and $H_4(X(S))\cong0$. For $i=1, 2$, we have the following sequence:
$$0\to H_2(X(S))\to H_2(W^{''})\to H_2(W^{''}, X(S))\to H_1(X(S))\to 0,$$ where $[S]$ is mapped to $0$ under $H_2(W^{''})\to H_2(W^{''}, X(S))$. Thus, $H_1(X(S))\cong H_2(X(S))\cong \mathbb{Z}$.\

Now, we compute the homology of $W$ using the Mayer-Vietoris sequence for the pair $(X(S), D^3\times S^1)$. In the long exact sequence
\begin{alignat*}{2}
\cdots & \to H_{i}(S^2\times S^1)\xrightarrow{\Phi_i} H_{i}(X(S)) & {}\oplus H_{i}(D^3\times S^1) & \to H_{i}(W) \\
       & \to H_{i-1}(S^2\times S^1)\to\cdots  & \cdots          & \to H_{0}(W)\to 0,
\end{alignat*}
$\Phi_i$ is injective for $i=0, 1$, and bijective for $i=2, 3$, which implies that $W$ has a homology of $D^3\times S^1$.
\end{proof}

\begin{proof}[Proof of Theorem \ref{main}]
Let $(M, \alpha)$ be a distinguished homology handle with trivial Alexander polynomial. Then by above lemmas, there is a topological 4-dimensional manifold $W$ whose homology is isomorphic to that of $S^1\times D^3$, and whose boundary is $M$. Choose  a cohomology class $\varphi \in H^1(W)$ whose restriction to M is dual to $\alpha$. By \cite[Assertion 5]{M1}, the infinite cyclic covering $\widetilde{W}_{\varphi}$ has finitely generated homology groups over $\mathbb{Q}$ since $W$ has the homology of the circle. Thus the pair $(W, \varphi)$ is a null $\widetilde{H}$-cobordism of $(M, \alpha)$.

\end{proof}

\end{document}